\documentclass[nospthms]{svjour3}
\smartqed

\usepackage{amsmath,amssymb}
\usepackage[hypertexnames=false]{hyperref}
\usepackage{cleveref}
\usepackage{autonum}
\usepackage[margin=2.5cm]{geometry}
\usepackage[dvipsnames,svgnames]{xcolor}
\usepackage{enumitem}
\usepackage{amsthm}

\newtheorem{thm}{Theorem}[section]
\crefname{thm}{Theorem}{Theorems}
\newtheorem{lem}{Lemma}[section]
\crefname{lem}{Lemma}{Lemmas}

\theoremstyle{definition}
\newtheorem{assump}{Assumption}
\crefname{assump}{Assumption}{Assumptions}
\newtheorem{rem}{Remark}
\crefname{rem}{Remark}{Remarks}
\newtheorem{defi}{Definition}
\crefname{defi}{Definition}{Definitions}

\newcommand{\bC}{\mathbb{C}}
\newcommand{\bR}{\mathbb{R}}
\newcommand{\cT}{\mathcal{T}}
\newcommand{\re}{\operatorname{Re}}
\newcommand{\dual}[3]{{}_{#3'}\langle #1, #2 \rangle_{#3}}
\DeclareMathOperator{\dv}{div}

\allowdisplaybreaks

\title{On the finite element approximation for non-stationary saddle-point problems
}

\titlerunning{FEM for non-stationary saddle-point problems}        

\author{Tomoya Kemmochi}

\institute{T. Kemmochi \at
              Graduate School of Mathematical Sciences, The University of Tokyo \\
              3-8-1 Komaba, Meguro-ku, Tokyo, 153-8914, Japan \\
              Tel.: +81-3-5465-7001\\
              \email{kemmochi@ms.u-tokyo.ac.jp}
}

\date{Received: date / Accepted: date}

\begin{document}

\maketitle

\begin{abstract}
In this paper, we present a numerical analysis of the hydrostatic Stokes equations, 
which are linearization of the primitive equations describing the geophysical flows of the ocean and the atmosphere.
The hydrostatic Stokes equations can be formulated as an abstract non-stationary saddle-point problem, 
which also includes the non-stationary Stokes equations.
We first consider the finite element approximation for the abstract equations with a pair of spaces under the discrete inf-sup condition.
The aim of this paper is to establish error estimates for the approximated solutions in various norms, 
in the framework of analytic semigroup theory.
Our main contribution is an error estimate for the pressure with a natural singularity term $t^{-1}$, 
which is induced by the analyticity of the semigroup.
We also present applications of the error estimates for the finite element approximations of the non-stationary Stokes and the hydrostatic Stokes equations.
\keywords{%
hydrostatric Stokes equations 
\and 
primitive equations 
\and 
non-stationary Stokes equations 
\and 
finite element method 
\and 
analytic semigroup
}
\subclass{%
65M60
\and 
76M10
\and 
76D07
}
\end{abstract}

\section{Introduction}

Let $\Omega = (0,1)^2 \times (-D,0) \subset \bR^3$ be a (shallow) box domain with $D>0$.
We consider the (non-stationary) hydrostatic Stokes equations 
\begin{equation}
\begin{cases}
u_t - \Delta u + \nabla_H p = 0, & \text{in } \Omega \times (0,T), \\
\dv_H \bar{v} = 0, & \text{in } \Omega \times (0,T), \\
u(0) = u_0, & \text{in } \Omega
\end{cases}
\label{eq:h-stokes-intro}
\end{equation}
for an unknown velocity $u \colon \Omega \to \bR^2$ and pressure $p \colon G:=(0,1)^2 \to \bR$, where
$\nabla_H p = (\partial_x p, \partial_y p)^T$, $\dv_H v = \partial_x v_1 + \partial_y v_2$,
and $\bar{v} = \int_{-D}^0 v(\cdot,z)dz$.
We impose the boundary conditions
\begin{equation}
\begin{cases}
\partial_z u = 0, & \text{on } \Gamma_u := G \times \{ 0 \}, \\
u = 0, & \text{on } \Gamma_b := G \times \{ -D \}, \\
\text{$u$ and $p$ are periodic} & \text{on } \Gamma_l := \partial G \times (-D,0).
\end{cases}
\label{eq:h-stokes-bc-intro}
\end{equation}
These are the linearized equations of the primitive equations (without the Coriolis force) described as
\begin{equation}
\begin{cases}
\partial_t u + (U \cdot \nabla) u - \Delta u + \nabla_H p = 0, & \text{in } \Omega \times (0,T), \\
\partial_z p = 0, & \text{in } \Omega \times (0,T), \\
\dv U = 0, & \text{in } \Omega \times (0,T), \\
u(0) = u_0, & \text{in } \Omega
\end{cases}
\label{eq:primitive}
\end{equation}
with boundary conditions
\begin{equation}
\begin{cases}
\partial_z u = 0, u_3 = 0, & \text{on } \Gamma_u, \\
U = 0, & \text{on } \Gamma_b, \\
\text{$U$ and $p$ are periodic} & \text{on } \Gamma_l,
\end{cases}
\label{eq:primitive-bc}
\end{equation}
where $U = (u, u_3) \colon \Omega \to \bR^3$.
The primitive equations are derived from the Navier-Stokes equations under the assumption that the vertical motion is much smaller than the horizontal motion, 
and were first introduced by Lions, Temam, and Wang~\cite{LioTW92,LioTW92b,LioTW95}.
This model is considered to describe the geophysical flows of the ocean and the atmosphere.

In this paper, we consider the finite element approximation of the hydrostatic Stokes problem \eqref{eq:h-stokes-intro} and \eqref{eq:h-stokes-bc-intro} as follows. 
Find $u_h \colon (0,T) \to V_h$ and $p_h \colon (0,T) \to Q_h$ satisfying the variational equation
\begin{equation}
\begin{cases}
(u_{h,t}, v_h)_\Omega + (\nabla u_h, \nabla v_h)_\Omega - (\dv_H \bar{v}_h, p_h)_G = 0, & \forall v_h \in V_h, \\
(\dv_H \bar{u}_h, q_h)_G = 0, & \forall q_h \in Q_h, \\
(u_h(0), v_h)_\Omega = (u_0,v_h)_\Omega, & \forall v_h \in V_{h,\sigma},
\end{cases}
\end{equation}
where $V_h \subset H^1(\Omega)^2$ and $Q_h \subset L^2_0(G)$ are finite-dimensional subspaces with suitable boundary conditions,
the bracket $(\cdot, \cdot)_X$ expresses the $L^2$-inner product over a domain $X$,
and $V_{h,\sigma} = \{ v_h \in V_h \mid (\dv_H \bar{v}_h, q_h)_G = 0,  \forall q_h \in Q_h \}$.
The precise definitions are given in Section~\ref{sec:application}.
The aim of this paper is to establish error estimates for $u_h$ and $p_h$ in the framework of analytic semigroup theory, as preliminaries for the numerical analysis of the primitive equations.
Although there are several results available on the finite element method for the steady hydrostatic Stokes equations (e.g.,~\cite{GuiR15,GuiR15b,GuiR16b,GuiR16}), there are no results for the non-stationary case, to the best of our knowledge.

As with the Navier-Stokes equations, the primitive equations are widely used in numerical computations for atmospheric and oceanic phenomena.
Finite element approximations and error estimates are presented in \cite{ChaR05,GuiR05,ChaGS12,ChaGGS15} for the steady primitive equations and in \cite{ChaR04,HeZXC16,GuiR17} for non-stationary problems.
In \cite{HeZXC16} and \cite{GuiR17}, error estimates for various fully discretized schemes are provided; these estimates have exponential growth in time $T$ (i.e., $e^{cT}$), since their arguments are based on the discrete Gronwall inequality.
Therefore, these results are time-local estimates in a sense.

We are interested in time-global error estimates for finite element approximations of the hydrostatic Stokes and the primitive equations.
In contrast to the three-dimensional Navier-Stokes equations, it is known that the three-dimensional primitive equations are globally well-posed in the $L^p$-settings (see~\cite{CaoT07} for $p=2$, and \cite{HieK16} for $p \in (1,\infty)$).
In their proofs, the analyticity of the hydrostatic Stokes semigroup plays a crucial role. 
We can thus expect that the analytic semigroup approach presents an efficient method for the numerical analysis of the primitive equations.
Indeed, for the two-dimensional Navier-Stokes equations, time-global error estimates have been obtained via the analytic semigroup approach in~\cite{Oka82b}.
Their results are based on the error estimates for the non-stationary Stokes equations established in \cite{Oka82}.

In order to derive error estimates for the finite element approximation, we formulate the hydrostatic Stokes equations as an abstract evolution problem.
We define
\begin{equation}
V = \{ v \in H^1(\Omega)^2 \mid v|_{\Gamma_b} = 0, \text{$v$ is periodic on $\Gamma_l$} \}.
\label{eq:V}
\end{equation}
Additionally, let $Q = L^2_0(G)$ and $H=L^2(\Omega)^2$.
Then, a weak form of the hydrostatic Stokes equations \eqref{eq:h-stokes-bc-intro} can be given as a non-stationary saddle-point problem as follows.
Find $u \colon (0,T) \to V$ and $p \colon (0,T) \to Q$ satisfying
\begin{equation}
\begin{cases}
(u_t(t), v)_H + a(u(t),v) + b(v,p(t)) = 0 , & \forall v \in V, \\
b(u(t),q) = 0, & \forall q \in Q,
\end{cases}
\label{eq:saddle-intro}
\end{equation}
where
\begin{equation}
a(u,v) = \iiint_\Omega \nabla u : \nabla v \,dxdydz,
\quad 
b(v,q) = - \iint_G (\dv_H \bar{v}) q \,dxdy
\end{equation}
for $u,v \in V$ and $q \in Q$.
Once we write the hydrostatic Stokes equations as above, 
we can use the same arguments for error estimates as is the case for the usual Stokes problem in~\cite{Oka82}.
Then, we can obtain the following error estimates for the velocity:
\begin{align}
\| \nabla (u(t) - u_h(t)) \|_{L^2(\Omega)} &\le C h t^{-1} \|u_0\|_{L^2(\Omega)} , \\
\| u(t) - u_h(t) \|_{L^2(\Omega)} &\le C h^2 t^{-1} \|u_0\|_{L^2(\Omega)},
\end{align}
under the assumptions stated in Section~\ref{sec:preliminary} (\cref{assump:inf-sup,assump:resolvent}).
In particular, the discrete inf-sup condition \eqref{eq:disc-infsup} plays an important role as in the stationary case.

According to \cite{Oka82}, we can also obtain an error estimate for the pressure.
The estimate presented in \cite{Oka82} is
\begin{equation}
\| p(t) - p_h(t) \|_{L^2} \le Ch (t^{-1} + t^{-3/2}) \|u_0\|_{L^2}
\end{equation}
for the two-dimensional Stokes equations
\begin{equation}
\begin{cases}
u_t - \Delta u + \nabla p = 0, \\
\dv u = 0, \\
u(0) = u_0
\end{cases}
\label{eq:stokes-intro}
\end{equation}
with the Dirichlet boundary condition.
However, the singularity $t^{-3/2}$ is unnatural from the viewpoint of analytic semigroup theory.
Indeed, an optimal order error estimate should be of the form
\begin{equation}
\| p(t) - p_h(t) \|_{L^2} \le C h \| \nabla p(t) \|_{L^2},
\end{equation}
and \eqref{eq:stokes-intro} yields
\begin{equation}
\| \nabla p(t) \|_{L^2} \le \| u_t(t) \|_{L^2} + \| \Delta u(t) \|_{L^2}
\le C t^{-1} \|u_0\|_{L^2}.
\end{equation}

In this paper, we first address the abstract problem \eqref{eq:saddle-intro} for bilinear forms $a \colon V \times V \to \bR$ and $b \colon V \times Q \to \bR$ defined on Hilbert spaces $V$ and $Q$,
and its Galerkin approximation
\begin{equation}
\begin{cases}
(u_{h,t}(t), v_h)_H + a(u_h(t),v_h) + b(v_h,p_h(t)) = 0 , & \forall v_h \in V_h, \\
b(u_h(t),q_h) = 0, & \forall q_h \in Q_h,
\end{cases}
\label{eq:disc-saddle-intro}
\end{equation}
for appropriate finite-dimensional subspaces $V_h \subset V$ and $Q_h \subset Q$.
The main contribution of this paper is to derive error estimates for the approximation problem \eqref{eq:disc-saddle-intro}, both for the velocity $u$ and the pressure $p$.
In particular, we remove the term $t^{-3/2}$ from the error estimate for the pressure.
Consequently, our results are a generalization and modification of the results in \cite{Oka82}.
After deriving the error estimates for \eqref{eq:disc-saddle-intro}, we apply the results to error estimates for the finite element approximation of the hydrostatic Stokes equations.

Here, we present the idea of the proof for error estimates in the case of the hydrostatic Stokes equations.
Our strategy is similar to that of \cite{Oka82}.
Namely, we first rewrite the error in terms of the contour integral of the resolvent and then we reduce the error estimate for the velocity to that of the resolvent problem.
The key idea is to establish the $V'$-error estimate for the resolvent problem as well as the $H^1$- and $L^2$-error estimates, which are already addressed in \cite{Oka82}.
Here, $V'$ denotes the dual space of $V$ defined by \eqref{eq:V}.
This estimate coincides with the $H^{-1}$-error estimate if we impose the Dirichlet boundary condition.
Then, we can obtain the $V'$-error estimate for the time derivative of the velocity of the form
\begin{equation}
\| u_t(t) - u_{h,t}(t) \|_{V'} \le Cht^{-1} \|u_0\|_{L^2(\Omega)}.
\end{equation}
Finally, we can establish an error estimate for the pressure without the term $t^{-3/2}$, with the aid of the discrete inf-sup condition.
We shall perform the above procedure in the abstract setting.

The rest of the paper is organized as follows.
In Section~\ref{sec:preliminary}, we introduce an abstract saddle-point problem and its Galerkin approximation, as well as the notation and assumptions in subsection~\ref{subsec:notation}.
After that, we introduce the resolvent problems and present some preliminary results in subsection~\ref{subsec:resolvents}.
Our main results are presented in Section~\ref{sec:abstract-error}.
As mentioned above, we derive the error estimate in the dual norm for the resolvent problem, and then we establish the error estimate for the evolution equation.
We apply these results for the Stokes and the hydrostatic Stokes equations in Section~\ref{sec:application}.
For the usual Stokes equations (subsection~\ref{subsec:stokes}), error estimates for the velocity are already available. 
However, the estimate for the pressure presented here is strictly sharper than that of \cite{Oka82}.
The error estimates for the non-stationary hydrostatic Stokes equations will be presented in subsection~\ref{subsec:h-stokes}.
Finally, we present our conclusions and areas for future works in Section~\ref{sec:conclusion}.

\section{Preliminaries}
\label{sec:preliminary}

\subsection{Notation and assumptions}
\label{subsec:notation}

Throughout this paper, except for in the last section, the symbols $H$, $V$, and $Q$ denote Hilbert spaces with dense and continuous injections $V \hookrightarrow H \hookrightarrow V'$, where $V'$ is the dual space of $V$, and $a \colon V \times V \to \bC$ and $b \colon V \times Q \to \bC$ are their continuous bilinear forms.
We assume that $a$ is symmetric for simplicity,
and that $a$ is coercive and $b$ satisfies the inf-sup condition:
\begin{gather}
\re a(v,v) \ge \alpha \|v\|_V, \quad \forall v \in V, \label{eq:coercive} \\
\sup_{v \in V} \frac{|b(v,q)|}{\|v\|_V} \ge \beta_1 \|q\|_Q, \quad \forall q \in Q, \label{eq:infsup}
\end{gather}
for some positive constants $\alpha$ and $\beta_1$.
We consider the following abstract non-stationary Stokes problem:
\begin{equation}
\begin{cases}
(u_t(t), v)_H + a(u(t),v) + b(v, p(t)) = 0, & \forall v \in V, \\
b(u(t), q) = 0, & \forall q \in Q, \\
u(0) = u_0,
\end{cases}
\label{eq:abstract-stokes}
\end{equation}
for $t \in (0,T)$, where $u_0 \in H_\sigma := \overline{V_\sigma}^{\| \cdot \|_H}$ and
\begin{equation}
V_\sigma := \{ v \in V \mid b(v,q) = 0, \ \forall q \in Q \}.
\end{equation}
We define a linear operator $A$ on $H_\sigma$ associated with the bilinear form $a$ by
\begin{equation}
\begin{cases}
D(A) = \{ u \in V_\sigma \mid \exists w \in H_\sigma \text{ s.t. } (w,v)_H = a(u,v), \ \forall v \in V_\sigma \}, \\
(Au, v)_H = a(u,v), \qquad \forall u \in D(A), \quad \forall v \in V_\sigma,
\end{cases}
\label{eq:operator-A}
\end{equation}
which is the abstract version of the Stokes operator.
By coercivity \eqref{eq:coercive} and the usual semigroup theory (e.g., see \cite{Paz83}), 
the operator $-A$ generates an analytic contraction semigroup $e^{-tA}$ on $H_\sigma$.
Thus, choosing $v \in V_\sigma$ as a test function in \eqref{eq:abstract-stokes}, 
we can construct a mild solution by $u(t) = e^{-tA}u_0$ for $t > 0$.
Moreover, owing to the inf-sup condition and the closed range theorem (see, e.g.,~\cite{ErnG04}), we can find $p(t) \in Q$ that satisfies
\begin{equation}
b(v,p(t)) = -(u_t(t), v)_H + a(u(t), v), \quad \forall v \in V,
\end{equation} 
for almost all $t \in (0,T)$. 
The uniqueness of these solutions is clear.
Finally, we define a linear operator $B \colon D(B) \subset Q \to H$ associated with $b$ by
\begin{equation}
\begin{cases}
D(B) = \{ q \in Q \mid \exists w \in H \text{ s.t. } (w,v)_H = b(v,q), \ \forall v \in V \}, \\
(Bq, v)_H = b(v,q), \qquad \forall q \in D(B), \quad \forall v \in V.
\end{cases}
\end{equation}

We next consider the Galerkin approximation for \eqref{eq:abstract-stokes}.
Let $V_h \subset V$ and $Q_h \subset Q$ be finite-dimensional subspaces.
We assume that they have the following properties:

\begin{assump}\label{assump:inf-sup}
\begin{enumerate}[label=(A-\arabic{*}),leftmargin=*]
\item{} [\textbf{discrete inf-sup condition}] There exists $\beta_2 > 0$ such that
\begin{equation}
\sup_{v_h \in V_h} \frac{|b(v_h,q_h)|}{\|v_h\|_V} \ge \beta_2 \|q_h\|_Q,
\quad \forall q_h \in Q_h,
\label{eq:disc-infsup}
\end{equation}
uniformly in $h>0$.
\item{} [approximation property (1)] For each $v \in D(A)$, we can find $v_h \in V_h$ satisfying
\begin{align}
\| v_h \|_V &\le C \|v\|_V, \\
\| v-v_h \|_H &\le C h \|v\|_V, \\
\| v-v_h \|_V &\le C h \|Av\|_H, \label{eq:A-1-3}
\end{align}
where $C$ is independent of $h$ and $v$.
\item{} [approximation property (2)] For each $q \in D(B)$, we can find $q_h \in Q_h$ satisfying
\begin{equation}
\| q-q_h \|_Q \le C h \|Bq\|_H,
\end{equation}
where $C$ is independent of $h$ and $q$.
\end{enumerate}
\qed
\end{assump}

\begin{rem}
The assumption (A-2) includes the condition on ``elliptic regularity''.
For example, let $\Omega \subset \bR^2$ be a polygonal domain, $H=L^2(\Omega)$, $V = H^1_0(\Omega)$, and $V_h$ be the conforming $P^1$-finite element space with respect to a shape-regular triangulation of $\Omega$.
Then, for every $v \in H^2(\Omega)$, we can construct $v_h \in V_h$ with the error estimate
\begin{equation}
\| v-v_h\|_{H^1_0} \le Ch \|v\|_{H^2}.
\end{equation}
However, if $A$ is the Laplace operator defined by
\begin{equation}
(Au, v)_{L^2} = (\nabla u, \nabla v)_{L^2}, \quad u,v \in H^1_0(\Omega),
\end{equation}
then the error estimate \eqref{eq:A-1-3} does not hold in general.
Indeed, \eqref{eq:A-1-3} requires that the condition $D(A) = H^2(\Omega) \cap H^1_0(\Omega)$ is satisfied, or equivalently,
\begin{equation}
\| v \|_{H^2} \le C \|Av\|_{L^2}
\end{equation}
for all $v \in D(A)$, which is not true for non-convex polygonal domains (see~\cite{Gri85}).
\qed
\end{rem}

We also introduce the discrete ``solenoidal'' space $V_{h,\sigma}$, defined as
\begin{equation}
V_{h,\sigma} := \{ v_h \in V_h \mid b(v_h,q_h) = 0, \ \forall q_h \in Q_h \}.
\label{eq:Vhsigma}
\end{equation}
Note that $V_{h,\sigma} \not\subset V_\sigma$ in general.
We now formulate the Galerkin (semi-discrete) approximation for the problem~\eqref{eq:abstract-stokes} as follows:
find $u_h(t) \in V_h$ and $q_h(t) \in Q_h$ that satisfy
\begin{equation}
\begin{cases}
(u_{h,t}(t), v_h)_H + a(u_h(t),v_h) + b(v_h, p_h(t)) = 0, & \forall v_h \in V_h, \\
b(u_h(t), q_h) = 0, & \forall q_h \in Q_h, \\
u_h(0) = P_{h,\sigma} u_0,
\end{cases}
\label{eq:disc-abstract-stokes}
\end{equation}
where $P_{h,\sigma} \colon H \to V_{h,\sigma}$ is the orthogonal projection.
We define the discrete ``Stokes operator'' $A_h$ by
\begin{equation}
(A_h u_h, v_h)_H = a(u_h, v_h), \quad \forall u_h, v_h \in V_{h,\sigma}.
\end{equation}
Then, as in the continuous case, the operator $-A_h$ generates an analytic contraction semigroup $e^{-tA_h}$ on $H_{h,\sigma} := (V_{h,\sigma}, \| \cdot \|_H)$, and thus we can construct a unique solution $(u_h, p_h)$ of the equation \eqref{eq:disc-abstract-stokes} due to the discrete inf-sup condition \eqref{eq:disc-infsup}.

\subsection{Finite element method for resolvent problems}
\label{subsec:resolvents}

We show error estimates for \eqref{eq:disc-abstract-stokes} via the resolvent estimates, as originally shown in~\cite{Oka82} for the non-stationary Stokes problem.
Let $\Gamma = \{ re^{\pm i(\pi - \delta)} \in \bC \mid r \in [0,\infty) \}$ be a path for $\delta \in (0,\pi/2)$, which is oriented so that the imaginary part increases along $\Gamma$.
Then, since the semigroups $e^{-tA}$ and $e^{-tA_h}$ are analytic, we can write
\begin{equation}
u(t) - u_h(t) 
= \frac{1}{2\pi i} \int_\Gamma e^{t\lambda} 
\left[ (\lambda + A)^{-1} - (\lambda + A_h) P_{h,\sigma} \right]
u_0 d\lambda.
\label{eq:duhamel}
\end{equation}
Therefore, the error estimate for $u_h$ is reduced to that of a resolvent problem:
\begin{equation}
\lambda (w,v)_H + a(w,v) = (g,v)_H, \quad \forall v \in V_\sigma,
\label{eq:resol}
\end{equation}
and
\begin{equation}
\lambda (w_h,v_h)_H + a(w_h,v_h) = (g,v_h)_H, \quad \forall v_h \in V_{h,\sigma},
\label{eq:disc-resol}
\end{equation}
for given $g \in H$ and $\lambda \in \Sigma_\delta := \{ z \in \bC \setminus \{0\} \mid |\arg z | < \pi - \delta \}$ with an arbitrarily fixed $\delta \in (0,\pi/2)$.
Owing to the closed range theorem, the problem \eqref{eq:resol} is equivalent to the following equation for $w$ and $\pi$:
\begin{equation}
\begin{cases}
\lambda (w,v)_H + a(w,v) + b(v,\pi) = (g,v)_H, & \forall v \in V, \\
b(w,q) = 0, & \forall q \in Q,
\end{cases}
\label{eq:resol-2}
\end{equation}
and \eqref{eq:disc-resol} is also equivalent to the problem 
\begin{equation}
\begin{cases}
\lambda (w_h,v_h)_H + a(w_h,v_h) + b(v_h,\pi_h) = (g,v_h)_H, & \forall v \in V_h, \\
b(w_h,q_h) = 0, & \forall q_h \in Q_h,
\end{cases}
\label{eq:disc-resol-2}
\end{equation}
where $w_h \in V_h$ and $\pi_h \in Q_h$ are unknown functions.
We assume that equation \eqref{eq:resol-2} admits the following estimate.
\begin{assump}\label{assump:resolvent}
(A-4) For each $g \in H$ and $\lambda \in \Sigma_\delta$, equation \eqref{eq:resol-2} has a unique solution $(w,\pi) \in V_\sigma \times Q$,
which admits the regularity $w \in D(A)$ and $\pi \in D(B)$.
Moreover, the following resolvent estimate holds:
\begin{equation}
|\lambda| \|w\|_H + |\lambda|^{1/2} \|w\|_V + \|Aw\|_H + \|B\pi\|_H \le C \|g\|_H,
\label{eq:resol-est}
\end{equation}
where $C$ is independent of $g$ and $\lambda$.
\qed
\end{assump}

It is known that assumptions (A-1)--(A-4) allow us to obtain error estimates for the velocity.
\begin{thm}\label{thm:resol-V}
Let $\delta \in (0,\pi/2)$, and
suppose that assumptions (A-1)--(A-4) hold.
Then, we have
\begin{align}
\left\| \left[ (\lambda + A)^{-1}P_\sigma - (\lambda + A_h)^{-1} P_{h,\sigma} \right] g \right\|_V 
&\le C h \|g\|_H, \label{eq:resol-error-1}\\
\left\| \left[ (\lambda + A)^{-1}P_\sigma - (\lambda + A_h)^{-1} P_{h,\sigma}  \right] g \right\|_H 
&\le C h^2 \|g\|_H, \label{eq:resol-error-2}
\end{align}
for any $g \in H$ and $\lambda \in \Sigma_\delta$, 
where $P_\sigma \colon H \to H_\sigma$ is the orthogonal projection and 
$C$ is independent of $h$, $g$, and $\lambda$.
\qed
\end{thm}
An outline of the proof is given below. 
We refer the reader to \cite[Theorems~3.2 and~4.2]{Oka82} for the detailed proof.
\begin{proof}
Fix $g \in H$, $\delta \in (0,\pi/2)$, and $\lambda \in \Sigma_\delta$ arbitrarily.
Let $(w,\pi)$ and $(w_h,\pi_h)$ be the solutions of \eqref{eq:resol-2} and \eqref{eq:disc-resol-2}, respectively.
Choose $v_h \in V_h$ and $q_h \in Q_h$ arbitrarily and consider $w_h - v_h$ and $\pi_h - q_h$.
Substituting $\phi_h \in V_h$ and $\psi_h \in Q_h$ into \eqref{eq:resol-2} and \eqref{eq:disc-resol-2}, 
we have
\begin{align}
a_\lambda(w_h-v_h,\phi_h) + b(\phi_h,\pi_h-q_h)
&= a_\lambda(w-v_h,\phi_h) + b(\phi_h,\pi-q_h) \\
&=: \dual{F_h}{\phi_h}{V_h}, \quad \forall \phi_h \in V_h,
\label{eq:thm21-1}
\end{align}
and
\begin{equation}
b(w_h-v_h,\psi_h) = b(w-v_h,\psi_h)
=: \dual{G_h}{\psi_h}{Q_h}, \quad \forall \psi_h \in Q_h,
\label{eq:thm21-2}
\end{equation}
where
\begin{equation}
a_\lambda(u,v) = \lambda(u,v)_H + a(u,v)
\end{equation}
for $u,v \in V$.
By the elementary inequality $|s\lambda + t| \ge \sin(\delta/2)(s|\lambda|+t)$ for $\lambda \in \Sigma_{\pi-\delta}$ and $s,t \ge 0$, we can obtain 
\begin{equation}
|a_\lambda(v,v)| \ge \alpha_1 \|v\|_{H^1(\Omega)}^2
\end{equation}
for all $v \in V$, where $\alpha_1>0$ is independent of $\lambda$.
Therefore, equations \eqref{eq:thm21-1} and \eqref{eq:thm21-2}, the discrete inf-sup condition (A-1), and the generalized Lax-Milgram theorem (e.g., see~\cite[Theorem~2.34]{ErnG04}) yield
\begin{equation}
|\lambda|^{1/2} \|w_h-v_h\|_H + \|w_h-v_h\|_V + \|\pi_h-q_h\|_Q
\le C \left( \|F_h\|_{V_h'} + \|G_h\|_{Q_h'} \right)
\label{eq:thm21-3}
\end{equation}
for some constant $C>0$, which is independent of $h$ due to (A-1).
From the definition of $F_h$ and $G_h$, we have
\begin{equation}
\|F_h\|_{V_h'} + \|G_h\|_{Q_h'}
\le C \left( |\lambda|^{1/2} \|w-v_h\|_H + \|w-v_h\|_V + \|\pi-q_h\|_Q \right),
\end{equation}
which implies, together with \eqref{eq:thm21-3}, that
\begin{equation}
|\lambda|^{1/2} \|w-w_h\|_H + \|w-w_h\|_V + \|\pi-\pi_h\|_Q 
\le C \left( |\lambda|^{1/2} \|w-v_h\|_H + \|w-v_h\|_V + \|\pi-q_h\|_Q \right).
\end{equation}
Therefore, assumptions (A-2)--(A-4) lead to
\begin{equation}
|\lambda|^{1/2} \|w-w_h\|_H + \|w-w_h\|_V + \|\pi-\pi_h\|_Q 
\le C h \|g\|_H,
\label{eq:thm21-5}
\end{equation}
which implies \eqref{eq:resol-error-1}.

The error estimate \eqref{eq:resol-error-2} is demonstrated by the standard duality argument.
Consider the dual problem 
\begin{equation}
\begin{cases}
\lambda (\phi, \zeta)_H + a(\phi, \zeta) + b(\phi, \eta) = (\phi, w-w_h)_H, & \forall \phi \in V, \\
b(\zeta, \psi) = 0, & \forall \psi \in Q.
\end{cases}
\label{eq:thm21-4}
\end{equation}
The solution $(\zeta,\eta) \in V \times Q$ has the estimate
\begin{equation}
|\lambda| \|\zeta\|_H + |\lambda|^{1/2} \|\zeta\|_V + \|A\zeta\|_H + \|B\eta\|_H \le C \|v-v_h\|_H,
\label{eq:resol-est-2}
\end{equation}
from assumption (A-4).
Substituting $\phi = w-w_h$ into \eqref{eq:thm21-4} and recalling equations \eqref{eq:resol-2} and \eqref{eq:disc-resol-2}, we have
\begin{equation}
\|w-w_h\|_H^2
\le Ch \|g\|_H \times \left( 
|\lambda|^{1/2} \|\zeta-\zeta_h\|_H + \|\zeta-\zeta_h\|_V + \|\eta-\eta_h\|_Q 
 \right)
\end{equation}
for any $\zeta_h \in V_h$ and $\eta_h \in Q_h$.
Hence, together with (A-2), (A-3), and \eqref{eq:resol-est-2}, we obtain \eqref{eq:resol-error-2}.
\end{proof}

\section{Abstract results}
\label{sec:abstract-error}

This section is devoted to the error estimates for the abstract Stokes problems.
\begin{thm}\label{thm:abstract}
Let $u_0 \in H_\sigma$ and let $(u,p)$ and $(u_h,p_h)$ be the solutions of \eqref{eq:abstract-stokes} and \eqref{eq:disc-abstract-stokes}, respectively.
Assume that (A-1)--(A-4) hold.
Then, we have the following error estimates:
\begin{align}
\| u(t) - u_h(t) \|_V &\le C h t^{-1} \|u_0\|_H, \label{eq:error-u-V} \\
\| u(t) - u_h(t) \|_H &\le C h^2 t^{-1} \|u_0\|_H, \label{eq:error-u-H} \\
\| u_t(t) - u_{h,t}(t) \|_{V'} &\le C h t^{-1} \|u_0\|_H, \label{eq:error-ut} \\
\| p(t) - p_h(t) \|_Q &\le C h t^{-1} \|u_0\|_H, \label{eq:error-p} 
\end{align}
for all $t \in (0,T)$, where each constant $C$ depends only on the constants appearing in assumptions (A-1)--(A-4), but is independent of $h$, $u_0$, $t$, and $T$.
\qed
\end{thm}

\begin{rem}
The error estimates for the velocity \eqref{eq:error-u-V} and \eqref{eq:error-u-H} are given in~\cite[Theorems~3.1 and~4.1]{Oka82}.
Although there is an error estimate for the pressure in~\cite{Oka82}, our result, \eqref{eq:error-p}, is strictly sharper.
Indeed, it was shown that the estimate
\begin{equation}
\| p(t) - p_h(t) \|_Q \le C h \left( t^{-1} + t^{-3/2} \right) \|u_0\|_H
\end{equation}
holds in~\cite[Theorem~5.1]{Oka82}.
\qed
\end{rem}

The estimates \eqref{eq:error-u-V} and \eqref{eq:error-u-H} are the consequence of the resolvent estimates \eqref{eq:resol-error-1} and \eqref{eq:resol-error-2},
and the estimate for the pressure \eqref{eq:error-p} is obtained from the discrete inf-sup condition \eqref{eq:disc-infsup}.
To show \eqref{eq:error-ut}, we need another resolvent estimate.

\begin{lem}\label{lem:error-ut}
Let $\delta \in (0,\pi/2)$.
Suppose that assumptions (A-1)--(A-4) hold.
Then, we have
\begin{equation}
\left\| \left[ (\lambda + A)^{-1}P_\sigma - (\lambda + A_h)^{-1} P_{h,\sigma}  \right] g \right\|_{V'} 
\le C h |\lambda|^{-1} \|g\|_H, \label{eq:resol-error-3}
\end{equation}
for any $g \in H$ and $\lambda \in \Sigma_\delta$,
where $C$ is independent of $h$, $g$, and $\lambda$.
\qed
\end{lem}

\begin{proof}
Fix $g \in H$, $\delta \in (0,\pi/2)$, and $\lambda \in \Sigma_\delta$ arbitrarily, and
let $(w,\pi)$ and $(w_h,\pi_h)$ be the solutions of \eqref{eq:resol-2} and \eqref{eq:disc-resol-2}, respectively.
It is sufficient to show that
\begin{equation}
(w-w_h,v)_H \le C h |\lambda|^{-1} \|g\|_H \|v\|_V
\label{eq:lem31-1}
\end{equation}
holds for arbitrary $v \in V$, 
since $\| F \|_{V'} = \sup_{v \in V} (F,v)_H / \|v\|_V$ for $F \in H \hookrightarrow V'$.
Fix $v \in V$ and choose $v_h \in V_h$ as in (A-2).
Then,
\begin{equation}
(w-w_h,v)_H = (w-w_h,v-v_h)_H + (w-w_h,v_h)_H =: I_1 + I_2.
\end{equation}
Since a standard energy method yields $\|w\|_H+\|w_h\|_H \le C|\lambda|^{-1}\|g\|_H$, we have
\begin{equation}
|I_1| \le C |\lambda|^{-1}\|g\|_H \cdot h \|v\|_V
\end{equation}
from assumption (A-2).
Moreover, equations \eqref{eq:resol-2} and \eqref{eq:disc-resol-2} imply
\begin{equation}
\lambda I_2 = -a(w-w_h,v_h) - b(v_h,\pi-\pi_h).
\end{equation}
Together with \eqref{eq:thm21-5}, we have
\begin{equation}
|I_2| \le Ch|\lambda|^{-1} \|g\|_H,
\end{equation}
which gives \eqref{eq:lem31-1}.
Hence we can complete the proof.
\end{proof}

\begin{proof}[Proof of \cref{thm:abstract}]
(1) Proof of \eqref{eq:error-u-V} and \eqref{eq:error-u-H}.
The derivation of these estimates was originally presented in \cite{Oka82}.
Indeed, one can check \eqref{eq:error-u-V} and \eqref{eq:error-u-H} directly from equation \eqref{eq:duhamel} and \cref{thm:resol-V}.

(2) Proof of \eqref{eq:error-ut}.
From \eqref{eq:duhamel} and \cref{lem:error-ut}, we have
\begin{equation}
\| u_t(t) - u_{h,t}(t) \|_{V'} 
\le \int_\Gamma |\lambda e^{t\lambda}| \cdot Ch |\lambda|^{-1} \|u_0\|_H |d\lambda| 
\le C h t^{-1} \|u_0\|_H, 
\end{equation}
where $\Gamma = \partial\Sigma_\delta$ for an arbitrary $\delta \in (0,\pi/2)$.

(3) Proof of \eqref{eq:error-p}.
Fix $q_h \in Q_h$ arbitrarily.
Then, by the discrete inf-sup condition \eqref{eq:disc-infsup}, we have
\begin{equation}
\beta_2 \| p_h(t) - q_h \|_Q
\le \sup_{v_h \in V_h} \frac{b(v_h,p_h(t) - q_h)}{\|v_h\|_V} .
\end{equation}
The equations \eqref{eq:abstract-stokes} and \eqref{eq:disc-abstract-stokes} yield
\begin{equation}
b(v_h, p_h(t) - q_h) = b(v_h,p(t)-q_h) + (u_t(t)-u_{h,t}(t),v_h)_H + a(u(t)-u_h(t),v_h),
\end{equation}
which leads to
\begin{equation}
\| p_h(t) - q_h \|_Q \le C \left( \| p(t)-q_h\|_Q + ht^{-1}\|u_0\|_H \right) \|v_h\|_V
\end{equation}
from \eqref{eq:error-u-V} and \eqref{eq:error-ut}.
Therefore, noting $p-p_h = (p-q_h) + (q_h-p_h)$, we have
\begin{equation}
\|p(t)-p_h\|_Q \le C \|p(t)-q_h\|_Q + Cht^{-1}\|u_0\|_H.
\end{equation}
Finally, choosing $q_h \in Q_h$ as in assumption (A-3), we obtain
\begin{equation}
\|p(t)-q_h\|_Q \le Ch\|Bp(t)\|_H \le C h t^{-1} \|u_0\|_H,
\end{equation}
since $Bp = -u_t - Au$ and $u(t) = e^{-tA}u_0$.
This completes the proof.
\end{proof}

\begin{rem}\label{rem:inhomogeneous}
Throughout this section, we have considered the homogeneous problem \eqref{eq:abstract-stokes}.
We now consider the inhomogeneous problem
\begin{equation}
\begin{cases}
(u_t,v)_H + a(u,v) + b(v,p) = \langle f,v \rangle_{V,V'}, & \forall v \in V, \\
b(u,q) = 0, & \forall q \in Q,
\end{cases}
\end{equation}
with external force $f \colon (0,T) \to V'$.
If $f \in C^\theta([0,T]; H)$ for some $\theta \in (0,1]$, then we have
\begin{equation}
\| u(t)-u_h(t) \|_H + h \| u(t)-u_h(t)\|_V
\le Ch \left( t^{-1} \| u_0 \|_H + t^\theta |f|_{C^\theta([0,T];H)} + \|f(t)\|_H \right)
\end{equation}
by the same argument as in \cite[\S~5]{FujM76}.
However, we cannot extend this result to the $V'$-error estimate and the pressure estimate at present.
\qed
\end{rem}

\section{Applications}
\label{sec:application}

In this section, we apply \cref{thm:abstract} to the non-stationary Stokes and the hydrostatic Stokes problem.
Hereafter, $L^2(\Omega)$ and $H^s(\Omega)$ denote the Lebesgue and Sobolev spaces, respectively.

\subsection{Non-stationary Stokes equation}
\label{subsec:stokes}

Let $\Omega \subset \bR^d$ ($d=2,3$) be a convex polygonal or polyhedral domain.
We consider the non-stationary Stokes equations in $\Omega$ with the homogeneous Dirichlet boundary condition:
\begin{equation}
\begin{cases}
(u_t(t), v) + (\nabla u(t), \nabla v) - (\dv v, p(t)) = 0, & \forall v \in H^1_0(\Omega)^d, \\
(\dv u(t), q) = 0, & \forall q \in L^2_0(\Omega), \\
u(0) = u_0,
\end{cases}
\label{eq:stokes}
\end{equation}
where $u_0 \in L^2_\sigma(\Omega)$.
Here, we use the usual notation
\begin{align}
H^1_0(\Omega) &= \{ v \in H^1(\Omega) \mid v|_{\partial\Omega} = 0 \}, \\
L^2_0(\Omega) &= \{ q \in L^2(\Omega) \mid \textstyle\int_\Omega q = 0 \}, \\
H^1_{0,\sigma}(\Omega) &= \{ v \in H^1_0(\Omega)^d \mid \dv v = 0 \}, \\
L^2_\sigma(\Omega) &= \overline{H^1_{0,\sigma}(\Omega)}^{\| \cdot \|_{L^2}},
\end{align}
and let $(\cdot, \cdot)$ denote the $L^2$-inner product over $\Omega$.
Let $H=L^2(\Omega)^d$, $V=H^1_0(\Omega)^d$, $Q=L^2_0(\Omega)$, $a(u,v)=(\nabla u, \nabla v)$, and $b(v,q)=-(\dv v, q)$.
Then, $H_\sigma = L^2_\sigma(\Omega)$ and $V_\sigma = H^1_{0,\sigma}(\Omega)$.
We also define the operators $A$ and $B$ as in Section~\ref{sec:preliminary}.
The domain of $A$ coincides with $H^2(\Omega) \cap V_\sigma$ and the estimate
\begin{equation}
\| v \|_{H^2} \le C \|Av\|_{L^2}, \quad \forall v \in D(A),
\label{eq:regularity-stokes}
\end{equation}
holds since $\Omega$ is a convex polygon or polyhedron 
(see~\cite{KelO76} for the 2D case and~\cite{Dau89} for the 3D case).
Hence, assumption (A-4) holds.

Next, we consider the finite element approximation of \eqref{eq:stokes}.
Let $\cT_h$ be a conforming triangulation (or tetrahedralization) of $\Omega$ with parameter $h=\max_{K \in \cT_h} \operatorname{diam} K$.
We define the pair of approximation spaces $(V_h, Q_h)$ as the $P^1$-bubble-$P^1$ element (MINI element) or $P^2$-$P^1$ element (Taylor-Hood) with respect to $\cT_h$.
For the precise definition, see, e.g.,~\cite{ErnG04}.
We set $V_{h,\sigma}$ as in \eqref{eq:Vhsigma}.
If $\cT_h$ is shape-regular and quasi-uniform, then we can check that assumptions (A-1)--(A-3) hold, together with \eqref{eq:regularity-stokes} (see,~e.g., \cite{BreS08,ErnG04}).

The approximation scheme is as follows.
Find $(u_h(t), p_h(t) \in V_h \times Q_h$ which satisfies
\begin{equation}
\begin{cases}
(u_{h,t}(t), v_h) + (\nabla u_h(t), \nabla v_h) - (\dv v_h, p_h(t)) = 0, & \forall v_h \in V_h, \\
(\dv u_h(t), q_h) = 0, & \forall q_h \in Q_h, \\
u_h(0) = P_{h,\sigma} u_0,
\end{cases}
\label{eq:disc-stokes}
\end{equation}
where $P_{h,\sigma}$ is the $L^2$-projection onto $V_{h,\sigma}$, as in \eqref{eq:disc-abstract-stokes}. 
Then, since we have already checked that assumptions (A-1)--(A-4) hold, we can state the following error estimates.

\begin{thm}\label{thm:stokes}
Let $\Omega$ be a convex polygonal or polyhedral domain, $\cT_h$ be a shape-regular and quasi-uniform triangulation of $\Omega$, and $(V_h,Q_h)$ be the pair of finite elements mentioned above.
Let $(u,p)$ and $(u_h,p_h)$ be the solutions of \eqref{eq:stokes} and \eqref{eq:disc-stokes}, respectively, for the initial value $u_0 \in L^2_\sigma(\Omega)$.
Then, we have the following error estimates:
\begin{align}
\| u(t) - u_h(t) \|_{H^1} &\le C h t^{-1} \|u_0\|_{L^2},\\
\| u(t) - u_h(t) \|_{L^2} &\le C h^2 t^{-1} \|u_0\|_{L^2}, \\
\| u_t(t) - u_{h,t}(t) \|_{H^{-1}} &\le C h t^{-1} \|u_0\|_{L^2}, \\
\| p(t) - p_h(t) \|_{L^2} &\le C h t^{-1} \|u_0\|_{L^2}, 
\end{align}
for all $t > 0$, where each constant $C$ is independent of $h$, $u_0$, $t$, and $T$.
\qed
\end{thm}

\begin{rem}\label{rem:optimal}
For the MINI element, the convergence rate is optimal.
However, for the Taylor-Hood element, it is not.
We leave the investigation of the optimal order error estimates for higher order elements as an area for future work.
\qed
\end{rem}

\subsection{Non-stationary hydrostatic Stokes equation}
\label{subsec:h-stokes}

The second example is the hydrostatic Stokes problem, which is a linearized form of the primitive equations.
Let $G=(0,1) \subset \bR^2$ and $\Omega = G \times (-D,0) \subset \bR^3$ with $D>0$.
The unknown functions of the hydrostatic Stokes equations are the horizontal velocity $u \colon \Omega \times (0,T) \to \bR^2$ and the surface pressure $p \colon G \times (0,T) \to \bR$, as follows.
\begin{equation}
\begin{cases}
u_t - \Delta u + \nabla_H p = 0,  & \text{in } \Omega \times (0,T), \\
\dv_H \bar{u} = 0, & \text{in } \Omega \times (0,T), \\
u(\cdot, 0) = u_0, & \text{in } \Omega,
\end{cases}
\label{eq:h-stokes}
\end{equation}
with boundary conditions
\begin{equation}
\begin{cases}
\partial_z u = 0, & \text{on } \Gamma_u := G \times \{ 0 \}, \\
u = 0, & \text{on } \Gamma_b := G \times \{ -D \}, \\
\text{$u$ is periodic} & \text{on } \Gamma_l := \partial G \times (-D,0),
\end{cases}
\label{eq:h-stokes-bc}
\end{equation}
where $\nabla_H q= (\partial_x q, \partial_y q)^T$, $\dv_H v = \partial_x v_1 + \partial_y v_2$, and $\bar{v}(x,y) = \int_{-D}^{0} v(x,y,z) dz$.

Let $H = L^2(\Omega)^2$, $V = \{ v \in H^1(\Omega)^2 \mid v|_{\Gamma_b} = 0 \text{ and } v|_{\Gamma_l} \text{ is periodic} \}$, $Q=L^2_0(G)$, $a(u,v) = (\nabla u, \nabla v)_\Omega$ for $u,v \in V$, and $b(v,q) = -(\dv_H \bar{v}, q)_G$ for $(v,q) \in V \times Q$, where $(\cdot, \cdot)_X$ denotes the $L^2$-inner product over a domain $X$.
Then, the weak formulation of the problem \eqref{eq:h-stokes} and \eqref{eq:h-stokes-bc} is described by equation \eqref{eq:abstract-stokes}, i.e., 
\begin{equation}
\begin{cases}
(u_t(t), v)_\Omega + (\nabla u, \nabla v)_\Omega - (\dv_H \bar{v}, p)_G = 0, & \forall v \in V, \\
(\dv_H \bar{u}, q)_G = 0, & \forall q \in Q.
\end{cases}
\label{eq:weak-h-stokes}
\end{equation}
Thus, we can construct a finite element scheme for the hydrostatic Stokes equations.
Let $\cT_h$ be a tetrahedralization (a set of open tetrahedra) of $\Omega$ with $h = \max_{K \in \cT_h} \operatorname{diam} K$, and let $\tilde{\cT}_h$ be the triangulation of $G$ induced by $\cT_h$.
Namely,
\begin{equation}
\tilde{\cT}_h = \{ T \subset G \mid \exists K \in \cT_h \text{ s.t.\ } \overline{T} = \overline{G} \cap \partial K \},
\end{equation}
where $\overline{T}$ and $\overline{G}$ are the closures in $\bR^2$.
We suppose that
\begin{itemize}[label=\textbullet]
\item $V_h$ is the space of $P^2$-finite elements or $P^1$-bubble finite elements with respect to $\cT_h$, and each $v_h \in V_h$ vanishes on $\Gamma_b$ and is periodic on $\Gamma_l$,
\item $Q_h$ is the space of $P^1$-finite elements with respect to $\tilde{\cT}_h$ and $\int_G q_h = 0$ for each $q_h \in Q_h$.
\end{itemize}
Then, we can introduce the finite element approximations for \eqref{eq:h-stokes} and \eqref{eq:h-stokes-bc} as follows.
Find $u_h \colon (0,T) \to V_h$ and $p_h \colon (0,T) \to Q_h$ which satisfy
\begin{equation}
\begin{cases}
(u_{h,t}(t), v_h)_\Omega + (\nabla u_h, \nabla v_h)_\Omega - (\dv_H \bar{v}_h, p_h)_G = 0, & \forall v_h \in V_h, \\
(\dv_H \bar{u}_h, q_h)_G = 0, & \forall q_h \in Q_h, \\
(u_h(0), v_h)_\Omega = (u_0, v_h)_\Omega, & \forall v_h \in V_{h,\sigma},
\end{cases}
\label{eq:disc-h-stokes}
\end{equation}
where $V_{h,\sigma}$ is defined by \eqref{eq:Vhsigma}.

In order to discuss the error estimates in the framework of \cref{thm:abstract}, we should check conditions (A-1)--(A-4).
Let $A$ be the operator defined by \eqref{eq:operator-A}, which is called the hydrostatic Stokes operator in the present case.
Then, it is known that $0 \in \rho(A)$, $A$ generates a bounded analytic semigroup on $H_\sigma$, and $A$ satisfies the regularity property
\begin{equation}
\| v \|_{H^2(\Omega)} \le C \|Av\|_{L^2(\Omega)}, \quad \forall v \in D(A).
\end{equation}
We refer to \cite[Theorem~3.1]{HieK16} for the proof.
Therefore, we can check that conditions (A-2)--(A-4) hold.

Finally, we confirm (A-1) holds by introducing a prismatic mesh.
\begin{defi}
We say that a tetrahedralization $\cT_h$ of $\Omega$ is \emph{prismatic} if the following condition holds: 
for each $K \in \cT_h$, there exists $T \in \tilde{\cT}_h$ such that
\begin{equation}
K \subset P_T := \{ (x,y,z) \in \Omega \mid (x,y,0) \in T \},
\end{equation}
where $\tilde{\cT}_h$ is the triangulation of $G$ induced by $\cT_h$.
\qed
\end{defi}
We can construct such a mesh by the following procedure.
\begin{enumerate}
\item Triangulate the surface $G$ and denote the triangulation by $\tilde{\cT}_h$.
\item Construct a prism $P_T$ in $\Omega$ for each $T \in \tilde{\cT}_h$.
\item Decompose each prism $P_T$ into tetrahedra so that the set of tetrahedra becomes a conforming tetrahedralization of $\Omega$.
\end{enumerate}
In \cite{ChaG00}, it is proved that the pair $(V_h,Q_h)$ mentioned above satisfies the discrete inf-sup condition \eqref{eq:disc-infsup}, provided that the mesh is prismatic.
Indeed, if the mesh is prismatic, then we can extend a function $q_h \in Q_h$ naturally to a piecewise linear function over the mesh $\cT_h$, and thus we can use the usual inf-sup condition for the MINI element or Taylor-Hood element.
Hence, we can confirm (A-1) holds.

Therefore, we can apply \cref{thm:abstract} and obtain the following error estimates.
\begin{thm}\label{thm:h-stokes}
Let $\Omega$, $\cT_h$, and $(V_h,Q_h)$ be as described above.
Assume that the mesh $\cT_h$ is shape-regular, quasi-uniform, and prismatic.
Let $(u,p)$ and $(u_h,p_h)$ be the solutions of \eqref{eq:h-stokes} and \eqref{eq:disc-h-stokes}, respectively, for the initial value $u_0 \in L^2(\Omega)^2$ satisfying $\dv_H \bar{u_0} = 0$ in the distributional sense.
Then, we have the following error estimates:
\begin{align}
\| u(t) - u_h(t) \|_{H^1(\Omega)} &\le C h t^{-1} \|u_0\|_{L^2(\Omega)},\\
\| u(t) - u_h(t) \|_{L^2(\Omega)} &\le C h^2 t^{-1} \|u_0\|_{L^2(\Omega)}, \\
\| u_t(t) - u_{h,t}(t) \|_{V'} &\le C h t^{-1} \|u_0\|_{L^2(\Omega)}, \\
\| p(t) - p_h(t) \|_{L^2(G)} &\le C h t^{-1} \|u_0\|_{L^2(\Omega)}, 
\end{align}
for all $t > 0$, where each constant $C$ is independent of $h$, $u_0$, $t$, and $T$, and $V'$ is the dual space of $V = \{ v \in H^1(\Omega)^2 \mid v|_{\Gamma_b} = 0 \text{ and } v|_{\Gamma_l} \text{ is periodic} \}$.
\qed
\end{thm}

\section{Concluding remarks}
\label{sec:conclusion}

In the present paper, we considered the abstract non-stationary saddle-point problem \eqref{eq:abstract-stokes} and its finite element approximation \eqref{eq:disc-abstract-stokes}.
Our main contribution (\cref{thm:abstract}) is the derivation of error estimates for the velocity and the pressure in various norms.
In particular, the error estimate for the pressure with the optimal singularity (i.e., the term $t^{-1}$) is a new result.
We then applied this result to establish error estimates for the finite element approximation for the non-stationary Stokes and the hydrostatic Stokes equations.
However, as mentioned in \cref{rem:inhomogeneous}, we have not obtained the error estimates for the pressure for inhomogeneous problems.
Moreover, the convergence rate is not optimal for finite elements of higher degree (\cref{rem:optimal}). Furthermore, we should consider a numerical analysis for the primitive equations \eqref{eq:primitive} in the framework of analytic semigroup theory, as performed in \cite{Oka82b} for the two-dimensional Navier-Stokes equations. These problems remain an area for future work.

\begin{acknowledgements}
The author would like to thank Professor Matthias Hieber for suggesting this topic and encouraging the author through valuable discussions during the author's stay in TU Darmstadt.
This work was supported by 
the Program for Leading Graduate Schools, MEXT, Japan,
JSPS KAKENHI grant number 15J07471,
JSPS-DFG Japanese-German Graduate Externship,
and IRTG~1529 on Mathematical Fluid Dynamics.
\end{acknowledgements}

\bibliographystyle{plain}

\end{document}